\documentclass[12pt,reqno]{amsart} 
\usepackage[utf8]{inputenc}
\usepackage[T1]{fontenc}
\usepackage{lmodern}
\usepackage{amsmath, amsthm, amssymb, amsfonts}
\usepackage{latexsym}
\usepackage{amsxtra} 
\usepackage[all]{xy}
\usepackage{nicefrac,mathtools}
\usepackage[shortlabels, inline]{enumitem}
\usepackage{microtype}
\usepackage{hyperref}
\hypersetup{colorlinks, linkcolor = blue}
\usepackage{graphicx,calc}
\usepackage{xcolor}
\usepackage{tikz-cd}
\usepackage{tikz}
\usepackage[T1]{fontenc}
\usepackage[utf8]{inputenc}
\usepackage{thmtools}
\usepackage{geometry}
\numberwithin{equation}{section}
\theoremstyle{plain}
\newtheorem{theorem}[equation]{Theorem}

\newtheorem{Lemma}[equation]{Lemma}

\newtheorem{proposition}[equation]{Proposition}

\newtheorem{corollary}[equation]{Corollary}

\theoremstyle{definition}
\newtheorem{definition}[equation]{Definition}
\newtheorem{notation}[equation]{Notation}

\theoremstyle{remark} \newtheorem{remark}[equation]{Remark}
\theoremstyle{remark} 
 \newtheorem*{remark*}{Remark}
\newtheorem*{remarks*}{Remark}

\newcommand{\mb}[1]{\mathbb{#1}}
\newcommand{\mc}[1]{\mathcal{#1}}
\newcommand{\V}[1]{\mathfrak{X}(#1)}
\newcommand{\w}{\omega}
\newcommand{\ku}{(\kappa,\mu)}

\title{On a metric symplectization of a contact metric manifold}
\author{Sannidhi Alape} \email{sannidhi.a@gmail.com}\address{Department of Mathematics, Indian Institute of Science Education and Research, Bhopal, India}
\keywords{Contact metric manifold, $\ku$-manifold, symplectization}
\subjclass[2020]{53C15, 53C25, 53D10}

\begin{document}
\begin{abstract} In this article, we investigate metric structures on the symplectization of a contact metric manifold and prove that there is a unique metric structure, which we call the metric symplectization, for which each slice of the symplectization has a natural induced contact metric structure. We then study the curvature properties of this metric structure and use it to establish equivalent formulations of the $\ku$-nullity condition in terms of the metric symplectization. We also prove that isomorphisms of the metric symplectizations of $\ku$-manifolds determine $\ku$-manifolds up to $\mc{D}$-homothetic transformations. These classification results show that the metric symplectization provides a unified framework to classify Sasakian manifolds, $K$-contact manifolds and $\ku$-manifolds in terms of their symplectizations.
\end{abstract}
\maketitle
\tableofcontents
\numberwithin{equation}{section}
\section{Introduction}\label{sec:intro}
The study of the interplay between contact manifolds and symplectic manifolds is an important topic in differential geometry. Symplectic manifolds with a boundary contact manifold, or more generally, symplectic manifolds with contact hypersurfaces have been studied extensively (see \cite{CTH}). In the contact metric category, the existence of almost contact metric hypersurfaces in symplectic manifolds with various metric structures has been carried out (see \cite{CHSK},\cite{CHSC},\cite{ACHS}). Sufficient conditions for the existence of a contact metric hypersurfaces have also been obtained. In this article, we present the following result, which is the metric analogue of a well known result in contact and symplectic topology (see Lemma \ref{lem:hypLVF}).
\begin{theorem}\label{thm:lvfcmm}
    Let $(B,\w)$ be a symplectic manifold with an associated metric $g$. Suppose there is a Liouville vector field $Y$ which is of unit length. Then, any hypersurface in $B$ which is orthogonal to $Y$ is a contact metric manifold.
\end{theorem}
We then focus our attention on metric structures on the symplectization of a contact metric manifold $(M,\eta, g,\phi)$. The symplectization is an important tool in the study of contact manifolds and contact metric manifolds. Many contact invariants are defined in terms of the symplectization (see \cite{SFT}). In the contact metric category, a particular class of manifolds, called the Sasakian manifolds, is defined using a natural almost complex structure on the symplectization. This natural almost complex structure has also been used by Boyer and Galicki in \cite{EMCG} to obtain a characterization of $K$-contact manifolds using properties of the symplectization. However, since the natural almost complex structure does not induce contact metric structures on the slices $(M\times \{t\})$, for all $t \in \mathbb{R}$, and since the symplectization of a contact manifold comes with a canonical choice of a Liouville vector field, we study metric structures on the symplectization which satisfy the hypothesis of Theorem \ref{thm:lvfcmm}, and present the following result which establishes uniqueness of the same under certain reasonable assumptions.
\begin{theorem}\label{thm:uniqmetsymp}
    There is a unique metric structure on the symplectization of a contact metric manifold $(M,\eta, g, \phi)$ satisfying the following conditions.
    \begin{itemize}
        \item The Liouville vector field $\partial_{t}$ is of unit length and orthogonal to $(M\times \{t_{0}\})$, for every $t_{0}\in\mathbb{R}$.
        \item The almost complex structure agrees with $\phi$ on the contact distribution induced on the hypersurfaces $M\times \{t\}$.
    \end{itemize}
\end{theorem}
The unique metric structure thus obtained will be called the \textit{metric symplectization} of a contact metric manifold. In addition to being contact metric hypersurfaces in the symplectization, we also observe that the slices $(M\times \{t\})$ are $\mc{D}$-homothetic transformations of the underlying contact metric manifold. Using the curvature relations of $\mc{D}$-homothetic transformations, and the fact that the projection map onto the $\mb{R}$-coordinate is a Riemannian submersion, we obtain the relationship between curvature of a contact metric manifold and its metric symplectization (see Theorem \ref{thm:currel}). As a corollary to the curvature relations that we obtain, we obtain the following result which specifies a constraint on the curvature of the metric symplectization of a contact metric manifold.
\begin{corollary}\label{cor:negriccur}
    The metric symplectization of a contact metric manifold cannot be of non-negative Ricci curvature. In particular, the Ricci curvature in the direction of the standard Liouville vector field is always $(-2n-4)$.
\end{corollary}

Next, we study the metric symplectizations of $\ku$-manifolds, a class of contact metric manifolds $(M,\eta,g,\phi)$ defined by Blair, Koufogiorgos and Papantoniou in \cite{KUM}, by the condition
\begin{equation}\label{eq:kum}
    R(X,Y)\xi = (\kappa I +\mu h)(\eta(Y)X-\eta(X)Y),
\end{equation}
where $\kappa,\mu\in\mb{R}$ are constants and $2h$ denotes the tensor field which is the Lie derivative of $\phi$. The authors in \cite{KUM} also establish numerous interesting properties for this class of manifolds. It is clear that $\ku$-manifolds are a generalization of manifolds satisfying the condition $R(X,Y)\xi = 0$, as well as a generalization of Sasakian manifolds. The authors prove that the condition specified in Equation \eqref{eq:kum} is invariant under $\mc{D}$-homothetic transformations. It is also proved that the curvature of a non-Sasakian $\ku$-manifold is completely determined by the constants $\kappa$ and $\mu$ (see Section \ref{subsec:kum} for more details). Following this, $\ku$-manifolds have been an important topic in contact metric geometry as they form a rich class of examples for $H$-contact manifolds(\cite{HCM}), locally $\phi$-symmetric manifolds (\cite{LPSCM}), bilegendrian manifolds (\cite{BLKU}) etc. In view of the fact that $\ku$-manifolds are classified locally up to $\mc{D}$-homothetic transformations (\cite{BI}) and the fact that the metric symplectization captures the entire $\mc{D}$-homothetic spectrum of a contact metric manifold, we present the following characterization of $\ku$-manifolds using curvature of their metric symplectizations.
\begin{theorem}\label{thm:currelkum}
    A contact metric manifold $(M,\eta, g,\phi)$ is a non-Sasakian $(\kappa,\mu)$-manifold if and only if there exist $\tilde{\kappa},\tilde{\mu}\in \mb{R}$ such that the metric symplectization $(M\times \mathbb{R}, \w, \bar{g}, J)$ satisfies the curvature relation
    \begin{equation}\label{eq:curclass}
        \mc{V}(R(X,Y)\xi_{t}) = (\tilde{\kappa}I + \tilde{\mu}h)(\eta_{t}(Y)X-\eta(X)Y),
    \end{equation}
    for some $t\in\mb{R}$ and for all $X,Y\in\V{M}$. Further, if Equation \eqref{eq:curclass} is satisfied for some $t\in\mb{R}$, then it is satisfied (for different constants $\tilde{\kappa},\tilde{\mu}$) for every $t\in\mb{R}$.
\end{theorem}
The curvature relations also give a geometric condition on the metric symplectization which forces the contact metric manifold to be $\eta$-Einstein (see Proposition \ref{prop:einstienmetsymp}). In particular, we have the following result which specifies sufficient conditions on the metric symplectization of a $\ku$-manifold which forces it to be rigid, i.e., Sasakian.
\begin{theorem}\label{thm:kumrig}
    If the metric symplectization of a $\ku$-manifold $M$ satisfies the curvature relation $\overline{Ric}=(-2n-4)dt^{2}$, then $M$ is Sasakian.
\end{theorem}
In the final section, we present an application of the metric symplectization which is motivated by the investigation of  characterization of contact manifolds using symplectomorphic symplectizations. We prove the following characterization for $\ku$-manifolds.
\begin{theorem}\label{thm:charkumsymp}
    Let $(M^{2n+1},\eta,g,\phi)$ be a $\ku$-manifold, and $(M'^{2n+1}, \eta', g',\phi')$ be a $(\kappa',\mu')$-manifold. Then, $M$ and $M'$ are isomorphic up to $\mc{D}$-homothetic transformations if and only if their metric symplectizations are isomorphic to each other.
\end{theorem}
\begin{remark}
    Theorem \ref{thm:charkumsymp} is interesting because it is in contrast with the result in the contact category where it is known (see \cite{SC}) that there exist contact manifolds which are not contactomorphic, whose symplectizations are symplectomorphic.
\end{remark}

The paper is organized as follows. In Section \ref{sec:prel}, fundamental notions in Riemannian contact and symplectic geometry relevant to the results in the article will be recalled. In Section \ref{sec:TMS}, we study contact metric hypersurfaces in certain symplectic metric manifolds and define the metric symplectization. In Section \ref{sec:CP}, we study the relations between the curvature of a contact metric manifold and its metric symplectization. We establish a characterizing condition for $\ku$-manifolds using the metric symplectization. We also prove a rigidity result for $\ku$-manifolds. In Section \ref{sec:app}, we give an application of the metric symplectization which relates isomorphisms of the symplectization to isomorphisms of the contact metric manifold up to $\mc{D}$-homothetic transformations.

\noindent {\bf Acknowledgements:} I am grateful to Atreyee Bhattacharya and Dheeraj Kulkarni for being available for discussions and providing useful comments and feedback. I also thank them for their detailed suggestions on the preliminary draft of the manuscript which has significantly improved the article.

The author was supported by the CSIR grant 09/1020(0143)/2019-
EMR-I, DST, Government of India.
\section{Preliminaries}\label{sec:prel}
\subsection{Contact and symplectic manifolds}\label{subsec:consym}
\begin{definition}[Contact manifolds]\label{def:con}
    A differential $1$-form $\eta$ on a manifold $M$ of dimension $(2n+1)$ is said to be a contact form if
    \begin{equation}\label{eq:con}
        \eta_{p} \wedge (d\eta_{p})^{n}\neq 0 \text{ for all }p\in M.
    \end{equation}
    In this case, the distribution $Ker(\eta)$ is called a contact distribution on the manifold $M$. A contact manifold is a pair $(M,\mc{D}=Ker(\eta))$ comprising an odd-dimensional manifold $M$ and a contact distribution $\mc{D}$.
\end{definition}
For a detailed treatment of contact manifolds, we refer the reader to \cite{GCT}.
\begin{remark}\label{rem:definingforms}
    If $\mc{D}=Ker(\eta)$ for some contact form $\eta$, then $\mc{D}=Ker(f\eta)$ for any non-vanishing function $f\in C^{\infty}(M)$. Thus, a contact distribution can be given by various contact forms. We will use the notation $(M,\eta)$ for a contact manifold, where $\eta$ is a contact form on $M$.
\end{remark}
\begin{definition}[Reeb vector field]\label{def:RVF}
    Given any contact manifold $(M,\eta)$, the following conditions determine a unique non-vanishing vector field $\xi$.
    \begin{itemize}
        \item $\eta(\xi) = 1$, and
        \item $d\eta(\xi,X) = 0$, for any $X\in\V{M}$.
    \end{itemize}
    This vector field is called the Reeb vector field associated to the contact form $\eta$. 
\end{definition}
Throught the article, $\xi$ will be used to denote the Reeb vector field associated to the contact manifold in consideration.

The even-dimensional counterparts of contact manifolds are symplectic manifolds, which we now define.
\begin{definition}\label{def:symp}
    A differential $2$-form $\w$ on a manifold $B$ of dimension is called a symplectic form if it is closed ($d\w = 0$) and
    \begin{equation}\label{eq:sympform}
        \w_{b}^{n}\neq 0 \text{ for all }b\in B.
    \end{equation}
    A symplectic manifold is a pair $(B,\w)$ comprising an even-dimensional manifold $B$ and a symplectic form $\w$ on $B$.
\end{definition}
Equivalence of contact (resp. symplectic) manifolds are given in terms of the contact (resp. symplectic) forms as follows.
\begin{definition}[Equivalence of Contact and Symplectic manifolds]\noindent\label{def:consympiso}
    \begin{itemize}
        \item Contact manifolds $(M_{1},\eta_{1})$ and $(M_{2},\eta_{2})$ are said to be contactomorphic to each other if there is a diffeomorphism $f:M_{1}\rightarrow M_{2}$ such that $f^{*}\eta_{2}=\eta_{1}$.
        \item Symplectic manifolds $(B_{1},\w_{1})$ and $(B_{2},\w_{2})$ are said to be symplectomorphic to each other if there is a diffeomorphism $f:B_{1}\rightarrow B_{2}$ such that $f^{*}\w_{2}=\w_{1}$.
    \end{itemize}
\end{definition}
The presence of certain types of vector fields on a symplectic manifold can be used to detect the existence of hypersurfaces (embedded submanifolds of codimension 1) which support an induced contact structure.
\begin{definition}\label{def:LVF}
    Let $(B,\w)$ be a symplectic manifold. A vector field $Y$ on $B$ is said to be a Liouville vector field if $\mc{L}_{Y}\w = \w$.
\end{definition}
\begin{Lemma}\label{lem:hypLVF}
    Let $(B,\w)$ be a symplectic manifold with a Liouville vector field $Y$. Then, any hypersurface which is transverse to $Y$ is a contact manifold with the induced contact form given by $\eta := i_{Y}\w = \w(Y,-)$.
\end{Lemma}
The interplay between symplectic manifolds and contact manifolds is a very prominent theme in the field and the above result is one instance of the same. The following is a way to associate a symplectic manifold to any contact manifold.
\begin{definition}\label{def:symplectization}
    Let $(M,\eta)$ be a contact manifold. On the product manifold $(M\times \mb{R})$, the $2$-form $\w := d(e^{2t}\eta)$ is a symplectic form, where $t$ denotes the coordinate on $\mb{R}$. The symplectic manifold $(M\times \mb{R}, \w)$ thus obtained is called the symplectization of the contact manifold $(M,\eta)$.
\end{definition}
It can be verified that the coordinate vector field $\partial_{t}$ is a Liouville vector field in the symplectization. For every  $t_{0}\in \mb{R}$, the slice $(M\times \{t_{0}\})$ is a hypersurface of $(M\times \mb{R})$ which is transverse to the Liouville vector field $\partial_{t}$. In view of Lemma \ref{lem:hypLVF}, the slice at $t_{0}$ is a contact manifold with the contact form $e^{2t}\eta$. In this article, we establish the analogue of Lemma \ref{lem:hypLVF} in the contact metric category.
\subsection{Contact metric manifolds}\hfill\label{subsec:cmm}

Riemannian geometry of contact (resp. symplectic manifolds) lies in the intersection of Riemannian geometry and contact (resp. symplectic geometry). In this subsection, we recall the notions from contact and symplectic metric geometry which will be used in the later sections. For a comprehensive treatment of Riemannian Geometry of Contact and Symplectic manifolds, see \cite{BRCS}.
\begin{definition}[Symplectic metric manifold]\label{def;sympmet}
    Let $(B,\w)$ be a symplectic manifold. A Riemannian metric $g$ is said to be compatible with $\w$ if there exists an almost complex structure $J$ (a $(1,1)$-tensor such that $J^{2} = -I$) which satisfies
    \begin{equation}\label{eq:sympcomp}
        g(X,JY) = \w(X,Y),
    \end{equation}
    for all $X,Y \in \V{B}$. In this case, the tuple $(B,\w, g,J)$ will be called a \textit{symplectic metric manifold}.
\end{definition}
\begin{remark}\label{rem:sympmetdef}
    For a symplectic metric manifold $(B,\w, g, J)$, any two of the structure tensors $\w, g$ and $J$ determine the other structure tensor. Equivalent definitions for symplectic metric manifolds can be given in terms of compatibility between any two of the structure tensors.
\end{remark}
It is known that a compatible Riemannian metric always exists for any symplectic manifold $(B,\w)$. Next, we recall the notion of compatibility between a contact form and a Riemannian manifold.
\begin{definition}[Contact metric manifold]\label{def:cmm}
    Let $(M,\eta)$ be a contact manifold with Reeb vector field $\xi$. A Riemannian metric $g$ is said to be compatible with $\eta$ if
    \begin{itemize}
        \item $g(X,\xi) = \eta (X)$, for all $X\in \V{M}$,
        \item There exists a $(1,1)$-tensor field $\phi$ on $M$ which satisfies $\phi^{2}=-I+\eta \otimes \xi$ and $d\eta(X,Y) = g(X,\phi Y)$, for all $X,Y\in \V{M}$.
    \end{itemize}
    In this case, the tuple $(M,\eta , g ,\phi )$ will be called a contact metric manifold.
\end{definition}
\begin{remark}\label{rem:consympmetiso}
    The notion of equivalence for the metric versions of contact (resp. symplectic) manifolds is given by contactomorphisms (resp. symplectomorphisms) which are isometries of the underlying Riemannian manifold. Equivalent contact or symplectic metric manifolds will be called \textit{isomorphic} to each other.
\end{remark}
On the symplectization of a contact metric manifold $(M,\eta, g,\phi)$, there is a natural almost complex structure $J$ defined by
\begin{equation}\label{eq:natalmcom}
    J(X,f\partial_{t})=(\phi X -f\xi, \eta(X)\partial_{t}),
\end{equation}
for all $X\in \V{M}$ and $f\in C^{\infty}(M\times \mb{R})$. It can be checked that $\bar{g}(\bar{X},\bar{Y}) = \w (J\bar{X},\bar{Y})$, for all $\bar{X},\bar{Y}\in \V{M\times \mb{R}}$ defines a Riemannian metric on the symplectization which is compatible with the symplectic form. This symplectic metric structure, which will be referred to as the \textit{natural symplectic metric structure} on the symplectization, has been used to define a particular class of contact metric manifolds by Shigeo Sasaki (see \cite{SSM}).
\begin{definition}[Sasakian manifold]\label{def:sm}
    A contact metric manifold is said to be Sasakian if the natural almost complex structure on its symplectization is integrable.
\end{definition}
As is evident from the definition, Sasakian manifolds are closely related to symplectic metric manifolds with integrable almost complex structures, which are more popularly known as Kähler manifolds. The relationship between Sasakian and Kähler manifolds has been studied and it is known that a Sasakian manifold is naturally sandwiched between two Kähler manifolds. We refer the reader to \cite{SG} for a detailed treatment of the same.

We now recall the definition of another class of contact metric manifolds, which generalize Sasakian manifolds.
\begin{definition}[$K$-contact manifold]\label{def:kcm}
    A contact metric manifold for which the Reeb vector field is a Killing vector field (the Reeb flow is by isometries) is called a $K$-contact manifold. 
\end{definition}
\begin{remark}\label{rem:tensorh}
    Given a contact metric manifold $(M,\eta,g,\phi)$, consider the tensor field $h:=\frac{1}{2}\mc{L}_{\xi}\phi$. It can be checked that the contact metric manifold is $K$-contact if and only if $h\equiv 0$. The tensor $h$ plays a very important role in the study of contact metric manifolds which are not $K$-contact.
\end{remark}
In \cite{EMCG}, Boyer and Galicki give a formulation of the $K$-contact condition using the natural symplectic structure on the symplectization $(M\times \mb{R}^{+},d(t\eta))$. 
\begin{theorem}[\cite{EMCG}]
    A compact contact metric manifold $(M,\eta, g, \phi)$ is a $K$-contact manifold if and only if its natural symplectization $(\bar{M},\w,\bar{g},J)$ is almost Kähler with $t\partial_{t} - i\xi$ pseudo-holomorphic.
\end{theorem}
It is clear, from the above illustrations, that contact metric structures are closely related to the symplectic metric structures on the symplectization. In this paper, we define an alternative symplectic metric structure on the symplectization in view of Theorem \ref{thm:lvfcmm} and study the influence of geometric properties of the metric symplectization on the contact metric manifold and vice versa.
\subsection{\texorpdfstring{Contact metric $\boldsymbol{(\kappa,\mu)}$-manifolds}{}}\hfill\label{subsec:kum}

In \cite{KUM}, Blair, Koufogiorgos and Papantoniou define and study a class of manifolds for which the Reeb vector field satisfies the nullity condition
\begin{equation}
    R(X,Y)\xi = (\kappa I +\mu h) ( \eta(Y) X - \eta (X) Y)
\end{equation}
This class of manifolds contains the class of manifolds satisfying the condition $R(X,Y)\xi = 0$, which has been an important class of contact metric manifolds, since it contains tangent sphere bundles of flat Riemannian manifolds as examples. However, the manifolds satisfying the condition $R(X,Y)\xi = 0$ are not stable under $\mc{D}$-homothetic transformations. By a $\mc{D}_{a}$-homothetic transformation of a contact metric manifold $(M,\eta, g,\phi)$, for a real number $a>0$, we mean a transformation of the structure tensors given by
\begin{equation}\label{eq:Dhomo}
    \eta' = a\eta, \qquad g'=ag+a(a-1)\eta \otimes \eta ,\qquad \phi'=\phi.
\end{equation}
It can be checked that the tuple $(M,\eta', g', \phi')$ is also a contact metric manifold with
\begin{equation}\label{eq:Dhomo2}
    \xi' = \frac{\xi}{a},\qquad h'=\frac{h}{a}.
\end{equation}
It can also be checked that the $\mc{D}_{a}$-homothetic transformation of a $\ku$-manifold is a $(\kappa ', \mu ')$-manifold, where
\begin{equation}\label{eq:kuDhomo}
    \kappa ' = \frac{\kappa+a^{2}-1}{a^2},\qquad \mu' = \frac{\mu+2a-2}{a}.
\end{equation}
Many interesting properties of $\ku$-manifolds have been given in \cite{KUM}, some of which will be recalled here.
\begin{theorem}[\cite{KUM}]
    Let $(M,\eta , g, \phi)$ be a $\ku$-manifold. Then,
    \begin{equation*}
        h^{2}=-(1-\kappa)\phi^{2}.
    \end{equation*}
    Thus, $\kappa \leq 1$. Moreover, $\kappa = 1$ if and only if $M$ is Sasakian. If $M$ is non-Sasakian, then $h$ has three eigenvalues 0, $\lambda$ and $-\lambda$, where $\lambda = \sqrt{1-\kappa}$. The tangent bundle of $M$ admits a decomposition into three mutually orthogonal and integrable distributions $\mc{D}_{\lambda}, \mc{D}_{-\lambda}$ and $\mc{D}_{0}$ given by the eigenspaces of $h$ corresponding to the eigenvalues $\lambda,\, -\lambda$ and 0, respectively. 
\end{theorem}
The following result establishes the fact that the curvature of a non-Sasakian $\ku$-manifold is completely determined by the constants $\kappa$ and $\mu$.
\begin{theorem}[\cite{KUM}]\label{thm:curkum}
    Let $(M^{2n+1}, \eta, g, \phi)$ be a non-Sasakian $\ku$-manifold. For $X_{\lambda}, Y_{\lambda}, Z_{\lambda} \in \mc{D}_{\lambda}$ and $X_{-\lambda}, Y_{-\lambda}, Z_{-\lambda} \in \mc{D}_{-\lambda}$, we have
    \begin{flalign}\label{eq:curkum}
        R(X_{\lambda}, Y_{\lambda})Z_{-\lambda} =& (\kappa - \mu)\left[ g(\phi Y_{\lambda}, Z_{-\lambda})\phi X_{\lambda} -g(\phi X_{\lambda}, Z_{-\lambda})\phi Y_{\lambda} \right]\notag \\
        R(X_{-\lambda}, Y_{-\lambda})Z_{\lambda} =& (\kappa - \mu)\left[ g(\phi Y_{-\lambda}, Z_{\lambda})\phi X_{-\lambda} -g(\phi X_{-\lambda}, Z_{\lambda})\phi Y_{-\lambda} \right]\notag \\
        R(X_{\lambda}, Y_{-\lambda})Z_{-\lambda} =& \kappa g(\phi X_{\lambda},Z_{-\lambda})\phi Y_{-\lambda} + \mu g(\phi X_{\lambda}, Y_{-\lambda})\phi Z_{-\lambda} \notag \\
        R(X_{\lambda}, Y_{-\lambda})Z_{\lambda} =& -\kappa g(\phi Y_{-\lambda},Z_{\lambda})\phi X_{\lambda} - \mu g(\phi Y_{-\lambda}, X_{\lambda})\phi Z_{\lambda} \notag \\
        R(X_{\lambda}, Y_{\lambda})Z_{\lambda} =& \left[2(1+\lambda)-\mu\right]\left[g(Y_{\lambda}, Z_{\lambda})X_{\lambda} - g(X_{\lambda}, Z_{\lambda})Y_{\lambda}\right] \notag \\
        R(X_{-\lambda}, Y_{-\lambda})Z_{-\lambda} =& \left[2(1-\lambda)-\mu\right]\left[g(Y_{-\lambda}, Z_{-\lambda})X_{-\lambda} - g(X_{-\lambda}, Z_{-\lambda})Y_{-\lambda}\right] 
    \end{flalign}
\end{theorem}
A classification of $3$-dimensional $\ku$-manifolds was also given (see \cite[Theorem 3]{KUM}). However, the local classification for all dimensions was given in \cite{BI}, where Boeckx defines a quantity called the index of a $\ku$-manifold $M$, which is denoted by $I_{M}$ and is given by
\begin{equation}\label{eq:BI}
    I_{M}=\frac{1-\frac{\mu}{2}}{\sqrt{1-\kappa}}.
\end{equation}
It was shown that the index is invariant under $\mc{D}$-homothetic transformations and that the index determines $\ku$-manifolds locally, in the following sense.
\begin{theorem}[\cite{BI}]\label{thm:BI}
    Let $(M_{i}, \eta_{i}, g_{i}, \phi_{i})$ be $(\kappa_{i},\mu_{i})$-manifolds for $i\in\{1,2\}$. Then, $I_{M_{1}} = I_{M_{2}}$ if and only if the two spaces are locally isometric up to a $\mc{D}$-homothetic transformation. In particular, if the spaces are simply connected and complete, then they are isometric up to a $\mc{D}$-homothetic transformation.
\end{theorem}
\subsection{Riemannian submersions}\hfill

The main results of the article involve curvature computations using Riemannian submersions. Here, we briefly recall the notions associated with the same. Let $\pi:M\rightarrow N$ be a smooth submersion. The Kernel of the differential of $\pi$ defines a distribution, which is known as the vertical distribution and is denoted by $\mc{V}$. If $M$ is a Riemannian manifold, then the orthogonal complement of $\mc{V}$ is called the horizontal distribution and is denoted by $\mc{H}$. In \cite{ONCUR}, O'Neill defined and studied Riemannian submersions.
\begin{definition}[\cite{ONCUR}]\label{def:Riemsub}
    A smooth submersion $\pi: (M,\overline{g}) \rightarrow (N,g)$ is said to be a Riemannian submersion if $d\pi$ is an isometry from $(\mc{H}_{x},\overline{g}_{x})$ to $(T_{\pi x}N,g_{\pi x})$, for every $x\in M$.
\end{definition}
\begin{notation}
    Throughout the article, in the context of a Riemannian submersion $\pi : (M,\bar{g}) \rightarrow (N,g)$, $\mc{H}(V)$ and $\mc{V}(V)$ will be used to denote the horizontal and vertical components of a vector field $V$. The quantities associated with the manifold $M$ will have an overhead line, the quantities associated with the fiber over an element $y\in N$ will have a superscript $y$. For example, $\overline{R}, R^{y}$ and $R$ will denote the curvature of the manifold $M$, the curvature of the fiber $\pi^{-1}(y)$, and the curvature of $N$, respectively.
\end{notation}
 The author in \cite{ONCUR} also defines the following tensor fields associated to a Riemannian submersion.
\begin{definition}[\cite{ONCUR}]\label{def:TandA}
    Given a Riemannian submersion $\pi : (M,\overline{g}) \rightarrow (N,g)$, define (2,1)-tensors $T$ and $A$ by the formulae
    \begin{flalign}
        T_{E_{1}}E_{2} &= \mc{H}\overline{\nabla}_{\mc{V}E_{1}}\mc{V}E_{2} + \mc{V}\overline{\nabla}_{\mc{V}E_{1}}\mc{H}E_{2}\label{eq:T}\\
        A_{E_{1}}E_{2} &= \mc{H}\overline{\nabla}_{\mc{H}E_{1}}\mc{V}E_{2} + \mc{V}\overline{\nabla}_{\mc{H}E_{1}}\mc{H}E_{2}\label{eq:A}
    \end{flalign}
    These tensor fields are called the fundamental tensor fields associated with the Riemannian submersion $\pi$.
\end{definition}
It is known that the tensor $T$ is identically zero if and only if each fibre is totally geodesic. On the other hand, the tensor $A$ is identically zero if and only if the horizontal distribution is integrable. Further, the following formulae can be used to compute $T$ and $A$.
\begin{proposition}\label{prop:TandAformula}
    Suppose $\pi: (M, \overline{g})\rightarrow (N,g)$ is a Riemannian submersion. For horizontal vector fields $X,Y$ and vertical vector fields $U,V$, we have 
    \begin{gather}
        T_{X}U = T_{X}Y= 0\\
        T_{V}U =T_{U}V = \mc{H}\overline{\nabla}_{U}V,\quad T_{U}X = \mc{V}\overline{\nabla}_{U}X\\
        A_{U}X = A_{U}V = 0 \\
        A_{X}Y = -A_{Y}X = \mc{V}\overline{\nabla}_{X}Y,\quad A_{X}U = \mc{H}\overline{\nabla}_{X}U
    \end{gather}
\end{proposition}
The relationship between the curvature of the domain and range of a Riemannian submersion can be formulated in terms of the fundamental tensors $T$ and $A$. Since the Riemannian submersion that we consider in this article has a one-dimensional base space, and integrable horizontal distribution, we recollect the formulae for curvature only for this specific case. For the complete set of relations in the general case, we refer the reader to \cite{ONCUR}.
\begin{theorem}[O'Neill's formulae for curvature (\cite{ONCUR})]\label{thm:curvoneill}
    Let $\pi: (M,\bar{g})\rightarrow (N,g)$ be a Riemannian submersion such that $\dim (N) =1$ and the horizontal distribution $\mc{H}$ is integrable. Then, for horizontal vector fields $X_{i}$ and vertical vector fields $V_{i}$ (for $i\in\{1,2,3,4\}$, and $y\in N$,
    \begin{flalign}\label{eq:curvoneill}
        (V_{1}, V_{2}, V_{3}, V_{4}) =& \bar{g}(R^{y}(V_{1}, V_{2})V_{3},V_{4}) + \bar{g}(T_{V_{1}}V_{3},T_{V_{2}}V_{4}) - \bar{g}(T_{V_{2}}V_{3},T_{V_{1}}V_{4}),\notag\\
        (V_{1},V_{2},V_{3},X_{1}) =& \Bar{g}((\Bar{\nabla}_{V_{1}}T)_{V_{2}}V_{3},X_{1}) - \Bar{g}((\Bar{\nabla}_{V_{2}}T)_{V_{1}}V_{3},X_{1})\notag \\
        (X_{1},V_{1},X_{2},V_{2}) =& \bar{g}(T_{V_{1}}X_{1} ,T_{V_{2}}X_{2})- \Bar{g}((\Bar{\nabla}_{X_{1}}T)_{V_{1}}V_{2},X_{2})  \notag\\
        (V_{1}, V_{2}, X_{1}, X_{2}) =& \bar{g}(T_{V_{1}}X_{1} ,T_{V_{2}}X_{2}) - \bar{g}(T_{V_{2}}X_{1} ,T_{V_{1}}X_{2})
    \end{flalign}
\end{theorem}
The Ricci curvature relations and the sectional curvature relations can be derived from this (see \cite{ONCUR}).
\section{The metric symplectization}\label{sec:TMS}
In this section, we will define the metric symplectization of a contact metric manifold. We begin with the proof of Theorem \ref{thm:lvfcmm}, which is the analogue of Lemma \ref{lem:hypLVF} in the contact metric category.

\begin{proof}[Proof of Theorem \ref{thm:lvfcmm}]
    Since $\omega$ is a symplectic form on $B$, there is an isomorphism
    \begin{flalign*}
        \Phi : TB &\rightarrow T^{*}B\\
        X &\mapsto (Y\mapsto \omega(Y,X))
    \end{flalign*}
    Let $M$ be a hypersurface in $B$ which is orthogonal to $Y$. Let $A_{Y}$ denote the covector field which takes value $1$ when evaluated on $Y$ and vanishes on vector fields in $M$. Using the fact that $Y$ is a Liouville vector field, it can be checked that the vector field $\Phi^{-1}(A_{Y})$ is the Reeb vector field for the induced contact structure $\eta := i_{Y}\w$ on $M$. This vector field will henceforth be denoted by $\xi$. Let $J$ denote the almost complex structure on $B$ given by the compatibility of $g$ and $\w$. In order to compute $J\xi$, note that, for any vector field $X$ in $M$,
    \begin{equation*}
        g(X,J\xi) = \omega(X,\xi) = A_{Y}(X)=0
    \end{equation*}
    Thus, $J\xi$ lies along the line field which is orthogonal to $M$. By our assumption of orthogonality, we must have that $J\xi$ lies along the direction of the vector field $Y$. Therefore, $J\xi = cY$, for some $c\in\mathbb{R}\setminus \{0\}$. In particular,
    \begin{equation*}
        g(\xi,\xi) = g(J\xi,J\xi)=c^{2}
    \end{equation*}
    On the other hand, the compatibility of the structure tensors of the symplectic metric structures gives us
    \begin{equation*}
        g(\xi,\xi) = -\omega(\xi,J\xi) = A_{Y}(J\xi) = c
    \end{equation*}
    From the equations above, we must have $c^{2}=c$. Since $c$ cannot be zero, we have $c=1$. Consequently, $J\xi = Y$. Since $J$ is an almost complex structure, we have $JY=-\xi$. For any vector field $X$ in the contact distribution, note that
    \begin{gather*}
        g(Y,JX) = \omega(Y,X) =i_{Y}\omega(X) = \eta (X) = 0\\
        g(\xi,JX)=-\omega(X,\xi)=-A_{Y}(X)=0
    \end{gather*}
    
    Thus, we can conclude that $J$ maps the contact distribution of $M$ to itself. Therefore, we can define a $(1,1)$-tensor field $\phi$ on $M$ which is equal to $J$ on the contact distribution and vanishes on the line field determined by $\xi$. We claim that $(M,\eta, \phi, g)$ is a contact metric structure. In order to see this, we need to verify the conditions specified in Definition \ref{def:cmm}. We have,
    \begin{equation*}
        g(X,\xi) = -\omega(\xi , JX) = A_{Y}(JX)
    \end{equation*}
    Using the values for $J$ obtained above, we can check that $g(X,\xi) = \eta (X)$ for all vector fields $X$ in $M$. The condition $\phi^{2} = -I + \eta \otimes \xi$ follows directly from the fact that $J$ is an almost complex structure and $\phi \xi = 0$. We also have, for $X,X'\in\V{M}$,
    \begin{equation*}
        d\eta(X,X')=(d\circ i_{Y}\omega)(X,X') = \omega(X,X') = g(X,JX')=g(X,\phi X')
    \end{equation*}
    Therefore, $(M,\eta, \phi, g)$ is a contact metric manifold.
\end{proof}
\begin{remark}\label{rem:hypcont}
    A standard example of contact hypersurfaces transverse to Liouville vector fields in symplectic manifolds is given in the context of the symplectization. Given the symplectization $(M\times \mathbb{R}, d(e^{2t}\eta))$, it is known that the coordinate vector field $\partial_{t}$ on $\mathbb{R}$ is a Liouville vector field. Consequently, the hypersurfaces $(M\times \{t_{0}\})$ are of contact type, with contact form $i_{\partial_{t}}\w$. However, in the metric category, the natural almost complex structure on the symplectization (see Equation \eqref{eq:natalmcom}) does not induce contact metric structures on the hypersurfaces $(M\times \{t_{0}\})$, for $t_{0}\neq 0$.
\end{remark} 

This motivates us to look for metric structures on the symplectization which satisfy the hypothesis of Theorem \ref{thm:lvfcmm}. To this end, we have Theorem \ref{thm:uniqmetsymp}, which is proved below.
\begin{proof}[Proof of Theorem \ref{thm:uniqmetsymp}]
    Let $(M\times \mb{R}, \w = d(e^{2t}\eta))$ be the symplectization of the contact manifold $(M, \eta)$. We know that the Liouville vector field $\partial_{t}$ induces a contact form $\eta_{t_{0}}=i_{\partial_{t}}\w $ on $(M\times \{t_{0}\})$ for every $t_{0}\in \mb{R}$. For $X\in\V{M}$,
    \begin{equation*}
        i_{\partial_{t}}\w(X) = \w(\partial_{t},X) = e^{2t_{0}}\eta(X)
    \end{equation*}
    Therefore, we have the contact form $\eta_{t}:= e^{2t}\eta$ on the hypersurface $M\times \{t\}$. The corresponding Reeb vector field is given by $\xi_{t}:= e^{-2t}\xi$.
    
    Suppose $J$ is an almost complex structure on $(M\times \mb{R})$ which is compatible with the symplectic form $\w$ and induces a Riemannian metric $\Bar{g}$ which satisfies the conditions specified in the statement of the theorem. Let $J\partial_{t} = X_{0} + f_{1}\xi + f_{2}\partial_{t}$, for some $X_{0}\in Ker(\eta_{t})$ and $f_{1}, f_{2} \in C^{\infty}(M\times \mb{R})$. By the assumption of ortogonality, we get 
    \begin{equation*}
        0=\Bar{g}(\partial_{t}, \xi) = \Bar{g}(J\partial_{t},J\xi)=\w(J\partial_{t}, \xi) = e^{2t}dt(J\partial_{t}) = f_{2}
    \end{equation*}
    On the other hand, for any vector field $X$ in the contact distribution,
    \begin{equation*}
        0=\Bar{g}(\partial_{t}, X) = \Bar{g}(J\partial_{t},JX)=\w(J\partial_{t}, X) = e^{2t}d\eta(J\partial_{t},X) = e^{2t}d\eta(X_{0},X)
    \end{equation*}
    Since the form $d\eta$ is non-degenerate on the contact distribution, we have $X_{0} = 0$. Thus, $J\partial_{t} = f_{1}\xi$ for some non-vanishing function $f_{1}\in C^{\infty}(M)$. Since $J$ is an almost complex structure, this also means that $J\xi = \frac{-\partial_{t}}{f_{1}}$. We now use the fact that $\partial_{t}$ is a Liouville vector field of unit length.
    \begin{equation*}
        1=\Bar{g}(\partial_{t},\partial_{t}) = \Bar{g}(J\partial_{t},J\partial_{t}) = \w(J\partial_{t},\partial_{t})=f_{1}\w(\xi,\partial_{t}) = -f_{1}e^{2t}
    \end{equation*}
    Hence, we can conclude that the only possible almost complex structure $J$ satisfying the required conditions is given by
    \begin{equation*}
        JX = \phi X\text{, for }X\in Ker(\eta_{t}),\qquad J\xi_{t}=\partial_{t},\qquad J\partial_{t}=-\xi_{t}.
    \end{equation*}
    It is easy to check the almost complex structure defined above is compatible with the symplectic form $\w$ and thus defines a symplectic metric structure on $(M\times \mb{R}, \w)$.
\end{proof}
The compatibility of the almost complex structure and the symplectic form uniquely determines the Riemannian metric $\bar{g}$ on the symplectization. Since the vector field $\partial_{t}$ is of unit length and orthogonal to the slices $(M\times \{t\})$, we have $\bar{g} = g_{t}+dt^{2}$, for a Riemannian metirc $g_{t}$ on the slice $(M\times \{t\})$. By Theorem \ref{thm:lvfcmm}, we can conclude that $(M\times \{t\}, \eta_{t}, g_{t}, \phi)$ is a contact metric manifold for every $t\in \mb{R}$. Using the compatibility conditions, we obtain
\begin{equation}\label{eq:slicemet}
    g_{t}=e^{2t}g+e^{2t}(e^{2t}-1)\eta \otimes \eta .
\end{equation}
Throughout the article $(M_{t}, eta_{t}, g_{t}, phi)$ will be used to denote the induced contact metric structure on $(M\times \{t\})$.
\begin{definition}\label{def:metsymp}
    Let $(M, \eta, g,\phi)$ be a contact metric manifold. The symplectic metric manifold $(M\times \mathbb{R},\w, \Bar{g}, \phi)$ as defined above is called the \textit{metric symplectization} of the given contact metric manifold.
\end{definition}
Next, we note that for $t \in \mathbb{R}$, the contact metric hypersurface $M\times \{t\}$ in the metric symplectization has structure tensors which is related to the original contact metric structure (or equivalently, the contact metric structure on $M\times \{0\}$) by
\begin{equation}\label{eq:dhomoslice}
    \eta_{t} = e^{2t}\eta, \qquad g_{t}=e^{2t}g+e^{2t}(e^{2t}-1)\eta \otimes \eta. 
\end{equation}
This proves that the slice at $t$ is a $\mc{D}_{e^{2t}}$-homothetic transformation of the original contact metric manifold (or the slice at $0$).
\begin{remark}\label{rem:metdhomo}
    More generally, for $t_{1}, t_{2}\in \mb{R}$, the slice at $t_{2}$ is a $\mc{D}$-homothetic transformation of the slice at $t_{1}$. Thus, the metric symplectization of a contact metric manifold captures the entire spectrum of its $\mc{D}$-homothetic transformations.
\end{remark}

Since $(\kappa, \mu)$-manifolds in particular are locally classified up to $\mc{D}$-homothetic transformations, we now focus on this class of contact metric manifolds and their metric symplectizations.
\section{Curvature properties of the metric symplectization}\label{sec:CP}
In this section, we will derive the relationship between the curvature of a contact metric manifold and the curvature of its metric symplectization. In particular, since the curvature of non-Sasakian $\ku$-manifolds are completely determined by the constants $\kappa$ and $\mu$ (see Theorem \ref{thm:curkum}), we obtain a characterization of the $\ku$ condition in terms of structure tensors of its metric symplectization. As an application of the curvature relations, we also present a rigidity result for $\ku$-manifolds under certain geometric constraints on the metric symplectization.

From the description of the Riemannian metric on the symplectization $\bar{g}=g_{t}+dt^{2}$, it is clear that the projection map $\pi_{2}: (M\times \mb{R},\bar{g}) \rightarrow (\mb{R},dt^{2})$ is a Riemannian submersion. In this case, the vertical distribution at a point $(x,t)$ is given by the tangent space $T_{x}M$ of $M$ and the horizontal distribution is given by the line field determined by the coordinate vector $\partial_{t}$. Using O'Neill's formulae (see \cite{ONCUR}), we can compute the curvature of the total space $M\times \mb{R}$ in terms of the curvature of the horizontal distribution, the curvature of the vertical distribution, and the fundamental tensors $T$ and $A$ associated with the Riemannian submersion $\pi_{2}$. The curvature of the horizontal distribution is zero since it is a one-dimensional space. Also, since the horizontal distribution is integrable for the Riemannian submersion $\pi_{2}$ under consideration, we have $A=0$.
\begin{Lemma}\label{lem:tensorT}
    The fundamental tensor $T$ associated with the Riemannian submersion $\pi_{2}$ satisfies the following relations. For $X,Y,Z\in \V{M}$,
    \begin{itemize}
        \item $T_{X}Y = -(\bar{g}(X,Y) + \eta_{t}(X) \eta_{t}(Y))\partial_{t}$,
        \item $T_{X}\partial_{t} = X+\eta_{t}(X)\xi_{t}$
    \end{itemize}
\end{Lemma}
The tensor $T$ is fully determined by the above result and the properties of $T$ stated in Proposition \ref{prop:TandAformula}. Using the values for obtained above and the curvature relations in Theorem \ref{thm:curvoneill}, we obtain the following result.
\begin{theorem}\label{thm:currel}
    The curvatures of the metric symplectization and the contact metric manifold satisfy the following relations.
    \begin{align}
        \mc{V}(\bar{R}(X,Y)Z) =& R^{t}(X,Y)Z+ g_{t}(X+\eta_{t}(X)\xi_{t},Z)Y \notag \\
        &- g_{t}(Y+\eta_{t}(Y)\xi_{t},Z)X + g_{t}(\eta_{t}(Y)X-\eta_{t}(X)Y,Z)\xi_{t}\label{eq:cur1}\\
        \mc{H}(\bar{R}(X,Y)Z) =& - g_{t}(\phi Z,\eta_{t}(Y) X - \eta_{t}(X) Y +\eta_{t}(Y) h X -\eta_{t}(X)  h Y)\notag \\
        &+2\eta_{t}(Z)g_{t}(Y,\phi X)\label{eq:cur2}\\
        \bar{g}(\bar{R}(\partial_{t},X)\partial_{t},Y)=&g_{t}(X,Y) + 3\eta_{t}(X)\eta_{t}(Y)\label{eq:cur3}\\
        \bar{g}(\bar{R}(X,Y)\partial_{t},\partial_{t}) =& 0\label{eq:cur4}
    \end{align}
\end{theorem}
The proof involves lengthy but straightforward computations using O'Neill's formulae and will be omitted.

In particular, for non-Sasakian $\ku$-manifolds, we use the $\ku$ condition \eqref{eq:kum} to prove Theorem \ref{thm:currelkum}.
\begin{proof}[Proof of Theorem \ref{thm:currelkum}]
    Suppose $(M,\eta, g, \phi)$ is a contact metric $\ku$-manifold. Substituting $Z=\xi_{t}$ in Equation \eqref{eq:cur1}, we see that
    \begin{align*}
        \mc{V}(\bar{R}(X,Y)\xi_{t}) =& R^{t}(X,Y)\xi_{t}+ g_{t}(X+\eta_{t}(X)\xi_{t},\xi_{t})Y \\
        &- g_{t}(Y+\eta_{t}(Y)\xi_{t},\xi_{t})X + g_{t}(\eta_{t}(Y)X-\eta_{t}(X)Y,\xi_{t})\xi_{t}\\
        =& (\kappa_{t}I+\mu_{t}h)(\eta_{t}(Y)X-\eta_{t}(X)Y)+2\eta_{t}(X)Y-2\eta_{t}(Y)X\\
        =& ((\kappa_{t}-2)I+\mu_{t}h)(\eta_{t}(Y)X-\eta_{t}(X)Y)
    \end{align*}
    Thus, Equation \eqref{eq:curclass} holds for all $t\in\mb{R}$ and $X,Y\in\V{M}$. For the converse, note that if Equation \eqref{eq:cur1} holds for the slice at $t$, then a computation similar to the one above shows that 
    \begin{equation*}
        R^{t}(X,Y)\xi_{t} = ((\tilde{\kappa}+2)I+\tilde{\mu}h)(\eta_{t}(Y)X-\eta_{t}(X)Y)
    \end{equation*}
    In other words, the slice at $t$ is a $(\tilde{\kappa}+2, \tilde{\mu})$-manifold. Since we have seen that every slice of the metric symplectization is a $\mc{D}$-homothetic transformation of every other slice, we can conclude that the slice at $t=0$ is a $\ku$-manifold. But the slice at $t=0$ is isomorphic to the contact metric manifold $M$. Therefore, $M$ is a $\ku$-manifold.

    The last statement is also a consequence of the observation that any two slices of the metric symplectization are $\mc{D}$-homothetic transformations of each other, and the fact that $\mc{D}$-homothetic transformation of a $(\kappa,\mu)$-manifold is a $(\kappa',\mu')$-manifold, for some $\kappa',\mu'\in\mb{R}$.
\end{proof}
As a corollary to Proposition \ref{thm:currel}, we obtain the relations between the Ricci curvature of a contact metric manifold and its metric symplectization.
\begin{corollary}\label{cor:Ricrel}
    Let $(M^{2n+1},\eta, g,\phi)$ be a contact metric manifold and $(M\times \mb{R}, \w , \bar{g}, J)$ be its metric symplectizaition. Let$\{e_{1},e_{2}, \dots ,e_{2n}\}$ be an orthonormal basis of the contact distribution of $M$. Then,
    \begin{gather*}
        \overline{Ric}(e_{i},e_{j}) = Ric^{t}(e_{i},e_{i}) - (2n+4)\delta_{ij},\\
        \overline{Ric}(e_{i},\xi_{t}) = Ric^{t}(e_{i},\xi_{t}), \quad \overline{Ric}(e_{i},\partial_{t}) = 0\\
        \overline{Ric}(\xi_{t},\partial_{t}) = 0,\quad \overline{Ric}(\xi_{t},\xi_{t}) = Ric^{t}(\xi_{t},\xi_{t})-4n-4\\
        \overline{Ric}(\partial_{t},\partial_{t}) = -2n-4
    \end{gather*}
\end{corollary}
Irrespective of the curvature of the contact metric manifold $M$, we notice that the Ricci curvature in the direction of $\partial_{t}$ is always equal to $(-2n-4)$, which proves Corollary \ref{cor:negriccur}.

Since the Ricci curvature in the $\partial_{t}$ direction is constant, we investigate the case in which the metric on the metric symplectization is Ricci flat in all other directions. Note that a contact metric manifold $(M,\eta, g,\phi)$ is said to be $\eta$-Einstein iff there exist $\alpha,\beta\in C^{\infty}(M)$ such that $Ric = \alpha g+\beta \eta \otimes \eta$.
\begin{proposition}\label{prop:einstienmetsymp}
    Let $(M, \eta, g,\phi)$ be a contact metric manifold. If the metric symplectization $(M\times \mb{R}, \w, \bar{g},J)$ satisfies the curvature relation $\overline{Ric} = (-2n-4)dt^{2}$, then $M\times \{t\}$ is $\eta_{t}$-Einstein, for all $t\in \mb{R}$. In particular, $M$ is $\eta$-Einstein.
\end{proposition}
\begin{proof}
    Using relations obtained in Corollary \ref{cor:Ricrel} and the basis, we have, for all $t\in\mb{R}$,
    \begin{equation}
        Ric^{t}(e_{i},e_{j}) = (2n+4)\delta_{ij}, Ric^{t}(e_{i},\xi_{t}) = 0, Ric^{t}(\xi_{t},\xi_{t}) = (4n+4)
    \end{equation}
    Since the basis $\{e_{i}\})_{i=1}^{n}$ is assumed to be orthonormal, we get
    \begin{equation}
        Ric^{t} = (2n+4)g_{t} + (4n+4)\eta_{t}\otimes \eta_{t}.
    \end{equation}
    Thus, the contact metric structure on the slice at $t$ of the metric symplectization is an $\eta_{t}$-Einstein manifold.
\end{proof}
In particular, in the class of $\ku$-manifolds, we have the rigidity result given by Theorem \ref{thm:kumrig}.
\begin{proof}
    It is known that a non-Sasakian $\ku$-manifold is $\eta$-Einstein if and only if $\mu = 2-2n$. Since Proposition \ref{prop:einstienmetsymp} forces the manifold $M$ and each of its $\mc{D}$-homothetic transformations to be $\eta_{t}$-Einstein, we need $\mu_{t} = 2-2n$ for all $t\in\mb{R}$. But it is also easy to check that the constant $\mu$ is invariant under $\mc{D}$-homothetic transformations if and only if $\mu = 2$. Thus, we must have $n=0$, which is not possible. Hence, the $\ku$-manifold is Sasakian.
\end{proof}
\section{An application of the metric symplectization}\label{sec:app}
In this section, we prove a result which characterizes $\ku$-manifolds using the metric symplectizations. For contact manifolds, the following result shows that characterization is not possible using symplectization.
\begin{theorem}[\cite{SC}]
    Let $M_{1}$ and $M_{2}$ be closed contact manifolds of same (odd) dimension greater than or equal to $5$ such that $\mb{R} \times M_{1}$ and $\mb{R}\times M_{2}$ are diffeomorphic. Then, for every contact structure $\mc{D}_{1}$ on $M_{1}$, there is a contact structure $\mc{D}_{2}$ on $M_{2}$ such that the symplectizations of $(M_{1},\mc{D}_{1})$ and $(M_{2},\mc{D}_{2})$ are symplectomorphic.
\end{theorem}
However, in the case of $\ku$-manifolds, we observe more rigidity, which is formulated in Theorem \ref{thm:charkumsymp} and proved below.
\begin{proof}[Proof of Theorem \ref{thm:charkumsymp}]
    If $M$ and $M'$ are $\mc{D}$-homothetic transformations of each other, then by the property \ref{rem:metdhomo}, we can conclude that a copy of $M$ is isomorphic to some slice $M'\times \{t'\}$ in the metric symplectization of $M'$. Consequently, the translation map $(x,t) \mapsto (x,t+t')$ can be checked to be a symplectomorphism of the metric symplectizations.

    For the converse, let $f:M'\times \mb{R}\rightarrow M\times \mb{R}$ be an isomorphism of the metric symplectizations. Let $\partial_{t'}$ and $\partial_{t}$ denote the natural Liouville vector fields on the domain and codomain, respectively. Suppose $df(\partial_{t'}) = Y$, where $Y\in\V{M\times \mb{R}}$. Since $f$ is an isomorphism and $\overline{Ric'}(\partial_{t'},\partial_{t'}) = -2n-4$, we must have $\overline{Ric}(Y,Y) = -2n-4$. The vector field $Y$ has an orthogonal decomposition
    \begin{equation}
        Y = Y_{\lambda_{t}}+Y_{-\lambda_{t}}+\alpha \xi_{t} + \beta \partial_{t},
    \end{equation}
    for $Y_{\lambda_{t}}\in\mc{D}_{\lambda_{t}}, Y_{-\lambda_{t}}\in\mc{D}_{-\lambda_{t}}$, and $\alpha, \beta \in C^{\infty}(M\times \mb{R})$.
    Using the formulae obtained for $\overline{Ric}$ in Corollary \ref{cor:Ricrel} and the fact that $Y$ is a vector field of unit length, we obtain
    \begin{flalign*}
        \overline{Ric}(Y,Y) =& \Vert Y_{\lambda_{t}} \Vert^{2}(Ric^{t}(\hat{Y}_{\lambda_{t}},\hat{Y}_{\lambda_{t}}) + \Vert Y_{-\lambda_{t}} \Vert^{2}(Ric^{t}(\hat{Y}_{-\lambda_{t}},\hat{Y}_{-\lambda_{t}})\\
        &+\alpha^{2}(Ric^{t}(\xi_{t},\xi_{t})-2n)-2n-4
    \end{flalign*}
    We can simplify the above equation using the Ricci curvature formulae for a $(\kappa_{t}, \mu_{t})$-manifold, and substituting $\mu_{t} = 2(1-\lambda_{t}I_{M_{t}})$ to get
    \begin{equation}\label{eq:curY}
        \begin{matrix}
            \Vert Y_{\lambda_{t}} \Vert^{2}(\lambda_{t}^{2}I_{M}-\lambda_{t}n(1+I_{M})+1)\\
            + \Vert Y_{-\lambda_{t}} \Vert^{2}(-\lambda_{t}^{2}+\lambda_{t}n(1-I_{M}) + 1)+\alpha^{2}n\lambda_{t}^{2}
        \end{matrix} =0
    \end{equation}
    The quadratic expressions $(\lambda_{t}^{2}I_{M}-\lambda_{t}n(1+I_{M})+1)$ and $(-\lambda_{t}^{2}+\lambda_{t}n(1-I_{M}) + 1)$ are positive for small values of $\lambda_{t}$. Since $\lambda_{t} = \frac{\lambda}{e^{2t}}$, we see that $\lambda_{t}$ is small for large values of $t$. In other words, there exists a $t_{0} \in \mb{R}$ such that both of the quadratic expressions are positive for $t > t_{0}$. Therefore, for $t>t_{0}$, Equation \eqref{eq:curY} can be satisfied only if $\Vert Y_{\lambda_{t}} \Vert = \Vert Y_{-\lambda_{t}} \Vert = \alpha = 0$. In this case, we will have $Y=\partial_{t}$. In this case, there exists $t'\in\mb{R}$ such that $f$ maps $M'\times \{t'\}$ to $M\times \{t\}$. Using the properties of the Liouville vector field, we can check that $f^{*}(\eta_t) = \eta_{t'}$. This shows that $f$ is a contactomorphism from $M'\times \{t'\}$ to $M\times \{t\}$, which is also a Riemannian isometry by definition. Hence, we can conclude that $f$, when restricted to $M'\times \{t'\}$, is an isomorphism  onto $M\times \{t\}$. Since $M$ and $M'$ are $\mc{D}$-homothetic transformations of these slices, we have the required result. 
\end{proof}
\begin{remark}\label{rem:unify}
    One can check that a contact metric manifold is Sasakian if and only if the almost complex structure corresponding to its metric symplectization is integrable. One can also check that the classification of $K$-contact condition using the natural almost complex structure given in \cite[Theorem 6.3]{EMCG} has an equivalent formulation in terms of the metric symplectization. Thus, the metric symplectization is a tool that can be used to classify Sasakian manifolds, $K$-contact manifolds, and non-Sasakian $\ku$-manifolds. 
\end{remark}
\bibliographystyle{plain}
\bibliography{References}
\vspace{10pt}
\end{document}